\documentclass[12pt]{article}

\usepackage[total={7in,8.75in},
top=1in, left=0.9in, includefoot]{geometry}

\usepackage{dsfont}

\usepackage{psfrag}

\usepackage{calrsfs}
\usepackage{mathrsfs}
\usepackage{pdfsync}
\usepackage{amsmath, amsthm, amssymb, amsfonts}
\usepackage{mathtools}

\usepackage[shortlabels]{enumitem}

\usepackage[longnamesfirst, comma]{natbib}

\setcitestyle{round}
\usepackage{url}

\usepackage{float}
\usepackage{multirow}

\usepackage[usenames]{color}

\usepackage{tikz}

\usetikzlibrary{arrows,decorations.pathmorphing,backgrounds,positioning,fit,automata}
\usetikzlibrary{petri}
\usetikzlibrary{topaths}

\usepackage{indentfirst,calc,euscript}
\usepackage{setspace}
\usepackage[reals]{layout}
\usepackage{xr}
\usepackage{amscd}

\usepackage{sgame}
\usepackage{graphicx}
\usepackage{subcaption}

\usepackage[colorlinks=true,breaklinks=true,bookmarks=true,urlcolor=magenta,
     citecolor=blue,linkcolor=red,bookmarksopen=false,draft=false]{hyperref}

\usepackage{colortbl}
\usepackage{authblk}

\newcommand{\diag}{\operatorname{diag}}

\newcommand{\Expect}{\operatorname{\mathsf{E}}}

\newcommand{\Var}{\operatorname{\mathsf{Var}}}
\newcommand{\Cov}{\operatorname{\mathsf{Cov}}}

\def\cprime{$'$}

\newcommand{\ignore}[1]{}

\newtheorem{theorem}{Theorem}

\newtheorem{proposition}[theorem]{Proposition}

\newtheorem{conjecture}[theorem]{Conjecture}

\theoremstyle{definition}

\newtheorem{definition}[theorem]{Definition}

\newtheorem{example}[theorem]{Example}

\newtheorem{remark}[theorem]{Remark}

\numberwithin{equation}{section}

\numberwithin{theorem}{section}

\begin{document}

\title{Variance Allocation and Shapley Value}
\author[1]{Riccardo Colini-Baldeschi\thanks{\texttt{rcolini@luiss.it}. This author is a member of is a member of GNCS-INdAM.}}
\author[1]{Marco Scarsini\thanks{\texttt{marco.scarsini@luiss.it}. This author is a member of GNAMPA-INdAM. His work is partially supported by PRIN 20103S5RN3 and MOE2013-T2-1-158.}}
\author[2]{Stefano Vaccari\thanks{\texttt{stefano.vaccari@uniroma1.it}.}}

\affil[1]{Dipartimento di Economia e Finanza, LUISS, Viale Romania 32, 00197 Roma, Italy}
\affil[2]{Dipartimento MEMOTEF, Sapienza-Universit\`a di Roma, Via Del Castro Laurenziano 9,  00161 Roma}

\maketitle

\begin{abstract}
Motivated by the problem of utility allocation in a portfolio under a Markowitz mean-variance choice paradigm, we propose an allocation criterion for the variance of the sum of $n$ possibly dependent random variables. This criterion, the Shapley value, requires to translate the problem into a cooperative game. The Shapley value has nice properties, but, in general,  is computationally demanding. The main result of this paper shows that in our particular case the Shapley value has a very simple form that can be easily computed. The same criterion is used also to allocate the standard deviation of the sum of $n$ random variables and a conjecture about the relation of the values in the two games is formulated.

\bigskip



\end{abstract}
\newpage

\section{Introduction}\label{se:intro}

In the mean-variance model of \citet{Mar:JF1952} preferences for prospects are represented by  a linear combination of the mean and the variance of the prospect:
\begin{equation}\label{eq:Utheta}
U_{\theta}[X]=\Expect[X]-\theta\Var[X],
\end{equation}
for some $\theta>0$ \citep[see, e.g.,][]{Ste:SIAMR2001}.

If an investor is dealing with a portfolio whose returns are $X_{1}, \dots, X_{n}$, it may be important for her to allocate to each asset of the portfolio its contribution to the total utility score $U_{\theta}[\sum_{i=1}^{n}X_{i}]$. 
Given that returns are typically correlated, allocating to asset $i$ a contribution equal to $\Expect[X_{i}]-\theta\Var[X_{i}]$ would not solve the problem. 
A similar problem is considered and solved (under some conditions) in the Capital Asset Pricing Model of \citet{Sha:JF1964}  and \citet{Lin:RES1965} \citep[see, e.g.,][]{FamFre:JEP2004}. 
In this paper we use tools of cooperative game theory to solve the problem.
In particular, we resort to the Shapley value \citep{Sha:CTG1953}, a standard solution concept for cooperative games, which has nice properties and has been extensively used in a variety of fields. 
First, we define a cooperative game where the players are the assets in the portfolio, and the worth of each coalition, i.e., of every subportfolio, is the utility score of this subportfolio. 
Then, we compute the Shapley value of each component of the portfolio. 

In general the Shapley value is used to allocate costs or gains among different agents who contribute to a joint project. The basic idea is that the allocation has to be fair for each player. In order to achieve this fairness each player is allocated a value that corresponds to her marginal contribution to the worth of a coalition, the average being suitably taken over all possible coalitions. 

Different solution concepts exist, based on different principles. For instance, the core, which embodies an idea of stability, is the set of all allocations such that no coalition has an incentive to deviate from the grand coalition. We refer the reader to 
\citet{PleSud:Springer2007} or \citet{MasSolZam:CUP2013} for a nice treatment of cooperative games and their solution concepts.

One of the main drawbacks of the Shapley value is its computational complexity as the number of players grows. This is due to the fact that its expression is an average over $n!$ marginal contributions.  In our problem, though, the expression of the Shapley value is extremely easy to compute and the complexity for the computation of the whole vector of Shapley values is quadratic in $n$. We will explain how this result fits into a more general result of \citet{ConSan:P19NCAI2004} about the Shapley value of decomposable games.

\subsection{Related literature}

We are not the first to propose game-theoretic tools for the analysis of cost allocations related to risks.

In an innovative paper in actuarial science, \citet{Lem:AB1984} proposes to use tools from cooperative game theory to allocate operating costs among different lines of an insurance company. The novelty of his paper is to solve a complicated accounting problem by computing a suitable solution of a cooperative game. Whereas the accounting problem is typically quite cumbersome, once the translation in game-theoretic language is performed, the solution is elegant and easy to interpret. Lemaire's paper gave rise to a whole literature on cost allocation in insurance. 
In this subsequent literature the attention is focused on the allocation of costs when dealing with a portfolio of risks. As \citet{Den:JR2001} points out, ``the problem of allocation is interesting and non-trivial because the sum of the risk capitals of each constituent is usually larger than the risk capital of the firm taken as a whole, something called the diversification effect. This decrease of total costs, or `rebate,' needs to be shared fairly between the constituents.''
His goal is to provide an allocation criterion that is based on fairness. 
Starting from the axiomatic definition of  coherent risk measures provided by \citet{ArtDelEbeHea:MF1999}, he proposes a set of axioms for the coherence of risk capital allocation principles. He ends up with an allocation that corresponds to the Aumann-Shapley value  of nonatomic cooperative games \citep[see][]{AumSha:Princeton1974}. 
\citet{TsaBar:IME2003} and \citet{Tsa:IME2004} propose the Aumann-Shapley value as an allocation mechanism when the risk measure is given by a distortion premium principle. 
\citet{Tsa:IME2009} does the same when convex risk measures are used.
\citet{AbbHos:SPL2013} use the Shapley value to allocate capital in the tail conditional expectation model.

Our analysis is close in spirit to a stream of literature in regression analysis that deals with quantifying the relative importance of each regressor for the response.
Contributions on this topic that make use of the Shapley value can be found in   \citet{LipCon:ASMBI2001, Lip:JMASM2006, Gro:AS2007, Gro:AS2009, GroLan:ASMBI2010, Mis:MPRA2016}, among others. These ideas actually go back to \citet{LinMerGol:SFC1980, Kru:AS1987}, who do not explicitly mention the connection with the Shapley value.
A parallel more general literature studies how to quantify the importance of random input variables to a function by using ANOVA techniques. 
One way to do it is through Sobol{\cprime} indices \citep{Sob:MM1990,Sob:MMCE1993} when the inputs are independent and their generalizations when the inputs are dependent \citep[see, e.g.,][]{ChaGamPri:EJS2012, ChaGamPri:JSCS2015, Owe:JUQ2013, GilPriArn:II2015}.
A different approach, based of the Shapley value has been proposed by \citet{Owe:JUQ2014} and further studied by \citet{SonNelSta:JUQ2016} and \citet{OwePri:arxiv2016}.

It is interesting to notice that the Shapley value has been employed in various problems that involve probability models. For instance, it has been 
implicitly used in reliability theory.  
\citet{BarPro:SPA1975} define an importance index of system components, whose $j$-th coordinate indicates the probability  that the failure of component $j$ causes the whole system to fail. \citet{MarMat:JMVA2013} point out that this index is actually a Shapley value.
Cooperative game theory tools have been used in queueing theory, see, e.g.,  
\citet{AniHav:OR2010}, among others, and in inventory, see, e.g.,
\citet{MueScaSha:GEB2002} and \citet{MonSca:GEB2007}.
We refer the reader to \citet{MorPat:TOP2008} for a nice survey of possible applications of the Shapley value in various fields.

Several authors have considered computational issues related to the Shapley value and have proposed efficient algorithms in some special cases, which typically involve voting games or games with a graph structure.
Among them \citet{DenPap:MOR1994},
\citet{ConSan:P19NCAI2004, IeoSho:P6CEC2005, IeoSho:P7CEC2006, FatWooJen:AI2008, FatWooJen:Springer2010, CasGomTej:COR2009, AzaChiChaPar:PICAA2014}. 
The interested reader should consult the book by 
\citet{ChaElkWoo:MCP2011} for a nice survey of computational aspect of cooperative game theory.

\subsection{Organization of the paper}
In Section~\ref{se:games} we introduce some fundamental concepts of cooperative game theory and some important solution concepts. In Section~\ref{se:variancegames} variance games are  defined and analyzed. 
Section~\ref{se:SDgames} deals with standard-deviation games and proposes a conjecture about the comparison between the two classes of games. 
Section~\ref{se:computational} deals with some computational aspects of the Shapley value and explains why the complexity of its calculation in the variance game is polynomial.

\section{Cooperative games}\label{se:games}

We start introducing some basic concepts in cooperative game theory.
Given a set of players $N=\{1, \dots, n\}$,  a \emph{cooperative game} is a pair $\langle N, \nu \rangle$, where  $\nu: 2^{N} \to \mathbb{R}$ is such that $\nu(\varnothing) = 0$. 
Any subset $J \subset N$ is called a \emph{coalition}. The set $N$ is called the \emph{grand coalition}. 
The function $\nu$ is called the \emph{characteristic function} of the game. Given that the set of players $N$ is fixed, for the sake of simplicity, we will just call game its characteristic function.
So, if $\nu$ represents utilities, then  $\nu(J)$ is the utility that  coalition $J$ can achieve by itself. If $\nu$ represents costs, then $\nu(J)$ is the cost that coalition $J$ must pay if it acts by itself.
We call $\mathcal{G}(N)$ the class of all games on $N$.

\subsection{Core and anticore}\label{suse:core}

A game $\xi$ is called \emph{additive} if for all $I, J \subset N$ such that $I \cap J = \varnothing$, we have
\[
\xi(I\cup J) = \xi(I) + \xi(J).
\]
We call $\mathcal{M}(N)$ the class of additive games on $N$. 

The \emph{core} (\emph{anticore}) of the game $\nu$ is defined as the set of all vectors $\boldsymbol{x}=(x_{1}, \dots, x_{n})$ such that
\begin{align}
&\nu(J) \le (\ge) \sum_{i\in J} x_{i} \quad \text{for all}\ J \subset N, \label{eq:coreineq}\\
&\nu(N) = \sum_{i\in N} x_{i} . \label{eq:coreeq}
\end{align}
A vector $\boldsymbol{x} \in \mathbb{R}^{n}$ such that \eqref{eq:coreeq} holds represents a possible allocation among players of what achieved by the grand coalition.
If $\nu$ represents a utility for coalitions, then the core is the set of all possible allocations that are stable, that is, no possible coalition has an incentive to deviate. Stability is given by \eqref{eq:coreineq}, which tells that no coalition by itself can achieve more than what allocated to it. If the function $\nu$ represents costs, the set of stable allocations is given by the anticore. The core (anticore) of a game can be empty.

A game $\nu$ is called \emph{supermodular} (\emph{submodular}) if for all $I, J \subset N$ 
\[
\nu(I \cup J) + \nu(I \cap J) \ge (\le) \nu(I) + \nu(J).
\]
It is well known \citep{Sha:IJGT1971} that the core of a supermodular game and the anticore of a submodular game are non-empty.

\subsection{Shapley value}\label{suse:shapley}

As seen in Subsection~\ref{suse:core}, the core is a set of stable allocations. Its appeal is the stability of its allocations. Among its shortcomings we have the fact that it may be empty and, when it is not a singleton, it is not clear how to choose one single suitable allocation. We now introduce a different solution concept that is obtained axiomatically and produces a single allocation.

Call player $i$ a \emph{dummy} if for all $J \subset N$ we have
\[
\nu(J \cup \{i\}) = \nu(J).
\]

Call players $i,j$ \emph{symmetric} if for all $J \subset N$ such that $i \not\in J$ and $j \not\in J$ we have
\[
\nu(J\cup\{i\}) = \nu(J\cup\{j\}).
\]

The \emph{Shapley value} of $\nu$ is a function $\boldsymbol{\phi} : \mathcal{G}(N) \to \mathbb{R}^{n}$ that satisfies the following properties:
\begin{enumerate}
\item\label{it:axEFF}
\emph{Efficiency}: $\sum_{i=1}^{n}\phi_{i}(\nu) = \nu(N)$.

\item\label{it:axSYM}
\emph{Symmetry}: If $i$ and $j$ are symmetric, then $\phi_{i}(\nu)=\phi_{j}(\nu)$.

\item\label{it:axDUM}
\emph{Dummy player}: If player $i$ is a dummy, then $\phi_{i}(\nu)=0$.

\item\label{it:axLIN}
\emph{Linearity}: for $\nu, \mu \in \mathcal{G}(N)$ and $\alpha, \beta \in \mathbb{R}$ we have
\[
\boldsymbol{\phi}(\alpha\nu+\beta\mu) = \alpha\boldsymbol{\phi}(\nu) + \beta\boldsymbol{\phi}(\mu).
\]
\end{enumerate}

\citet{Sha:CTG1953} showed that  the only function $\boldsymbol{\phi}$ with these four properties has the following form
\begin{align}
\phi_{i}(\nu) &= \sum_{J\subset N\setminus\{i\}}\frac{|J|!(N-|J|-1)!}{N!} (\nu(J\cup\{i\})-\nu(J)) \label{eq:shapley}\\
&= \frac{1}{n!}\sum_{\psi \in \mathcal{P}(N)}\left(\nu(P^{\psi}(i)\cup \{i\}) - \nu(P^{\psi}(i))\right), \label{eq:shapleyperm}
\end{align}
where $\mathcal{P}(N)$ is the set of all permutations of $N$,  $P^{\psi}(i)$ is the set of players who precede $i$ in the order determined by permutation $\psi$, and $|J|$ is the cardinality of $J$.

In general the Shapley value of a game does not necessarily lie in its core.
If the game is supermodular (submodular), the Shapley value lies in the core (anticore) and is actually its barycenter.

\subsubsection{Shapley fusion property}

Consider a game $\langle N, \nu \rangle$. For $J \subset N$ consider the new game $\langle N^{J}, \nu^{J} \rangle$ where all players in $J$ are fused into a single player. 
A game $\langle N, \nu \rangle$ such that, for all $J\subset N$, we have 
\begin{equation*}
\phi_{J}(\nu^{J}) = \sum_{i\in J}\phi_{i}(\nu).
\end{equation*}
is said to satisfy the \emph{Shapley fusion property}.

In general the Shapley fusion property does not hold, as the following counterexample shows. 

\begin{example}
Take $N\{1,2,3\}$ and, for every $i,j\in N$, with $i\neq j$,
\begin{align*}
\nu(\{i\})&=0,\\
\nu(\{i,j\})&=1\\
\nu(N)&=1.
\end{align*}
By symmetry, $\phi_{i}(\nu)=1/3$,  but, for $J=\{2,3\}$, we have
\begin{equation*}
\nu^{J}(\{1\})=0,\quad\nu^{J}(J)=1,\quad\nu^{J}(\{1,J\})=1,
\end{equation*}
hence $\phi_{J}(\nu^{J})=1\neq\phi_{2}(\nu)+\phi_{3}(\nu)$.

\end{example}

\section{Variance games} \label{se:variancegames}

We consider a random vector $\boldsymbol{X}:=(X_{1}, \dots,X_{n})$ whose components can be seen, for instance, as the returns of $n$ securities in a portfolio. The return of the whole portfolio is then $\sum_{i=1}^{n}X_{i}$. If preferences are represented by the utility score $U_{\theta}$ defined in \eqref{eq:Utheta}, we want to fairly allocate the utility $U_{\theta}[\sum_{i=1}^{n}X_{i}]$ of the portfolio to each of its components. To do this we have to take into account the correlation among the various returns. 
To achieve this goal we turn to cooperative game theory. We define a suitable cooperative game based on $U_{\theta}$ and we use the Shapley value of this game as the allocation criterion.

Consider a random vector $\boldsymbol{X}=(X_{1}, \dots, X_{n})$ with finite second moments and define, for every $J \subset N$ 
\begin{equation}\label{eq:SJ}
S_{J}:=\sum_{i \in J}X_{i}.
\end{equation}
For each $J\subset N$, define
\begin{equation}\label{eq:gamma}
\gamma(J):=U_{\theta}[S_{J}] = \Expect[S_{J}]-\theta\Var[S_{J}].
\end{equation}
The expression \eqref{eq:gamma} defines a cooperative game $\langle N, \gamma \rangle$. This game is a linear combination of two other games:
\begin{equation}\label{eq:gammaLC}
\gamma(\cdot) = \varepsilon(\cdot) - \theta \nu(\cdot),
\end{equation}
where
\begin{equation}\label{eq:nu}
\varepsilon(J):=\Expect[S_{J}]\quad\text{and}\quad \nu(J):=\Var[S_{J}].
\end{equation}
Given Property~\ref{it:axLIN} of the Shapley value, we have
\begin{equation*}
\boldsymbol{\phi}(\gamma) = \boldsymbol{\phi}(\varepsilon) - \theta\boldsymbol{\phi}(\nu).
\end{equation*}
Since the expectation is a linear operator, the game $\varepsilon$ is additive and  
\begin{equation*}
\phi_{i}(\varepsilon)=\Expect[X_{i}].
\end{equation*}
Therefore, the problem of finding the Shapley value of $\gamma$ reduces to finding the Shapley value of $\nu$, which we call a \emph{variance game}.

\subsection{Main result}

The next result shows that the Shapley value of the game $\nu$ has a very intuitive simple form in terms of the covariance matrix of $\boldsymbol{X}$. 

\begin{theorem}\label{th:shapley}
For $\nu$ defined as in \eqref{eq:nu} we have
\begin{equation}\label{eq:Shapleyalloc}
\phi_{i}(\nu) = \Cov[X_{i}, S_{N}] .
\end{equation}
\end{theorem}

\begin{proof}
From \eqref{eq:nu}, for $i \not\in J$, we have
\begin{align*}
\nu(J\cup\{i\})-\nu(J) &=
\Var\left(\sum_{j \in J\cup\{i\}} X_{j} \right) - \Var\left(\sum_{j \in J} X_{j} \right) \\
&= \sum_{j, \ell \in J\cup\{i\}} \Cov[X_{j},X_{\ell}] - \sum_{j, \ell \in J} \Cov[X_{j},X_{\ell}] \\
&=\Var[X_{i}] + 2\sum_{j\in J}\Cov\left[X_{i},X_{j}\right].
\end{align*}
Therefore, from \eqref{eq:shapleyperm} it follows that
\begin{align*}
\phi_{i}(\nu) &= \frac{1}{n!} \sum_{\psi \in \mathcal{P}(N)} \left(\Var[X_{i}] +2 \sum_{j \in P^{\psi}(i)}\Cov[X_{i}, X_{j}]\right) \\
&= \Var[X_{i}] +  \frac{2}{n!} \sum_{j\in N\setminus\{i\}} \sum_{\psi:j\in P^{\psi}(i)}\Cov[X_{i}, X_{j}] \\
&= \Var[X_{i}] +  \sum_{j\in N\setminus\{i\}} \Cov[X_{i}, X_{j}] \\
&= \sum_{j=1}^{n} \Cov[X_{i}, X_{j}] \\
&= \Cov[X_{i}, S_{N}]. 
\end{align*}
\qed
\end{proof}
The allocation in \eqref{eq:Shapleyalloc} is similar (although not equal) to the one obtained by \citet[][Theorem~3.2]{Wan:SCOR2002} for multinormally distributed risks, when the exponential tilting model is used. 
Analogously, \citet{OwePri:arxiv2016} find a similar allocation for a linear regression model with multinormally distributed regressors. 
Notice that Theorem~\ref{th:shapley} does not make any parametric assumption on the distribution of $\boldsymbol{X}$.

\begin{remark}
An immediate corollary of Theorem~\ref{th:shapley} is that the Shapley fusion property holds for variance games.
\end{remark}

\subsection{Additional properties}

We examine now some interesting properties of the variance allocation through the Shapley value.

\begin{example}
The Shapley value of the variance game can assume negative values. Consider the case $\boldsymbol{X}=(X_{1},X_{2})$ with $X_{2}=-2X_{1}$ and $\Var[X_{1}]=1$. Then 
\[
\Cov[\boldsymbol{X}] = 
\begin{pmatrix}
1 & -2\\
-2 & 4 
\end{pmatrix},
\]
hence $\phi_{1}(\nu) = -1$ and $\phi_{2}(\nu) = 2$. The idea is that, if a random variable contributes to hedge a risk, then it is ``rewarded'' with a negative Shapley value.
\end{example}

The following property says that, if perfect hedging can be achieved, then the Shapley value is identically zero, no matter what the individual variances are.

\begin{proposition}\label{pr:zerovar}
If $\Var[S_{N}] = 0$, then $\boldsymbol{\phi}(\nu) = \boldsymbol{0}$.
\end{proposition}

\begin{proof}
Let $\Var[S_{N}] = 0$, that is
\begin{equation}\label{eq:varSN}
\sum_{i\in N} \sum_{j\in N} \Cov[X_{i}, X_{j}] = 0.
\end{equation}
Then $S_{N}$ is almost surely a constant, which, without any loss of generality, we can assume to be zero. Hence for each $i \in N$
\[
X_{i} = - \sum_{j \in N\setminus\{i\}} X_{j},
\]
which implies
\begin{equation}\label{eq:varXi}
\Var[X_{i}] = \sum_{j \in N\setminus\{i\}}  \sum_{\ell \in N\setminus\{i\}} \Cov[X_{j}, X_{\ell}].
\end{equation}
Plugging \eqref{eq:varXi} into \eqref{eq:varSN} we obtain
\[
\sum_{j\in N} \Cov[X_{i}, X_{j}] = 0 \quad\text{for all }i\in N,
\]
which, by Theorem~\ref{th:shapley}, gives the desired result.
\qed
\end{proof}

As the following example shows, it is not possible to apply the result of Proposition~\ref{pr:zerovar} to a subvector of the vector $\boldsymbol{X}$. 

\begin{example}
It is possible to have $\Var[S_{J}] = 0$ for some $J \subset N$, without having $\phi_{j}(\nu) = 0$ for all $j\in J$. For instance, let $\boldsymbol{X}=(X_{1}, X_{2}, X_{3}, X_{4})$ be such that
\[
-X_{1} = X_{2} = X_{3} = X_{4},
\]
with $\Var[X_{1}] = 1$. Then 
\[
\Cov[\boldsymbol{X}] = 
\begin{pmatrix}
1 & -1 & -1 & -1 \\
-1 & 1 & 1 & 1 \\
-1 & 1 & 1 & 1 \\
-1 & 1 & 1 & 1 
\end{pmatrix},
\]
$\Var[X_{1}+X_{2}] = 0$ and $\phi_{1}(\nu) = -2$ and $\phi_{2}(\nu) = 2$. 

This is due to the fact that the Shapley value is computed globally, looking at the marginal contributions of a random variable to the variance of all possible subvectors of the vector $\boldsymbol{X}$. On the other hand, what is true is that, if $\Var[S_{J}] = 0$, then $\sum_{j\in J} \phi_{j} = 0$.
\end{example}

\begin{example}
Even if the Shapley value has the symmetry property, it is possible to have $X_{i}$ and $X_{j}$ exchangeable (or even i.i.d.) without necessarily having $\phi_{i}(\nu) = \phi_{j}(\nu)$. 
For instance, consider $\boldsymbol{X}=(X_{1}, X_{2}, X_{3}, X_{4})$ such that $X_{2}$ and $X_{3}$ are i.i.d. and
\[
X_{1} = X_{2}, \qquad X_{4} = -X_{3}.
\]
Let $\Var[X_{1}] = 1$. Then 
\[
\Cov[\boldsymbol{X}] = 
\begin{pmatrix}
1 & 1 & 0 & 0 \\
1 & 1 & 0 & 0 \\
0 & 0 & 1 & -1 \\
0 & 0 & -1 & 1 
\end{pmatrix}.
\]
Therefore $\phi_{2}(\nu) = 2$ and $\phi_{3}(\nu) = 0$.

Again, this is due to the global property of the Shapley value. Two exchangeable random variables can have very different relations with the other components of $\boldsymbol{X}$, therefore their Shapley value can differ.
\end{example}

Finally, we look at supermodularity (submodularity) of the variance game, which, as mentioned before, has important implications for the nonemptiness of its core (anticore).

\begin{proposition}\label{pr:core}
\begin{enumerate}[{\rm (a)}]
\item\label{it:pr:core-a}
If $\Cov[X_{i}, X_{j}] \ge 0$ for all $i,j \in N$, then the game $\nu$ is supermodular.

\item\label{it:pr:core-b}
If $\Cov[X_{i}, X_{j}] \le 0$ for all $i,j \in N$, then the game $\nu$ is submodular.
\end{enumerate}
\end{proposition}

\begin{proof}
\noindent
\ref{it:pr:core-a} 
If $\Cov[X_{i}, X_{j}] \ge 0$, then we have
\begin{align*}
\nu(I \cup J) + \nu(I \cap J) &= \Var[S_{I \cup J}] + \Var[S_{I \cap J}] \\
&= \sum_{i \in I \cup J} \sum_{j \in I \cup J} \Cov[X_{i}, X_{j}] + \sum_{i \in I \cap J} \sum_{j \in I \cap J}\Cov[X_{i}, X_{j}] \\
&= \sum_{i \in I} \sum_{j \in I} \Cov[X_{i}, X_{j}] + \sum_{i \in J} \sum_{j \in J} \Cov[X_{i}, X_{j}] + 2 \sum_{i\in I\setminus J} \sum_{j\in J\setminus I} \Cov[X_{i}, X_{j}] \\
&\ge \sum_{i \in I} \sum_{j \in I} \Cov[X_{i}, X_{j}] + \sum_{i \in J} \sum_{j \in J} \Cov[X_{i}, X_{j}] \\
&= \Var[S_{I}] + \Var[S_{J}] \\
&= \nu(I) + \nu(J).
\end{align*}

\noindent
\ref{it:pr:core-b} 
If $\Cov[X_{i}, X_{j}] \le 0$, then the inequality goes in the opposite direction.
\qed
\end{proof}

\section{Standard deviation games}\label{se:SDgames}

Given a random vector $\boldsymbol{X}=(X_{1}, \dots, X_{n})$ we can define a standard deviation game $\lambda$ on $N=\{1, \dots, n\}$ as follows:
\begin{equation*}
\lambda(J)=\sqrt{\Var[S_{J}]},
\end{equation*}
where $S_{J}$ is defined as in \eqref{eq:SJ}. Computing the Shapley value for this game is much more difficult than for the variance game. We will examine the relation between these two types of games.

The next example shows that the Shapley fusion property does not hold for standard deviation games.

\begin{example}
Consider the following covariance matrix
\begin{equation*}
\Sigma=
\begin{bmatrix}
1 & 0 & 0 \\
0 & 4 & 0 \\
0 & 0 & 9
\end{bmatrix}.
\end{equation*}
The corresponding standard deviation game is
\begin{align*}
&\lambda(\{1\})=1, \quad \lambda(\{2\})=2, \quad  \lambda(\{3\})=3,\\
&\lambda(\{1,2\})=\sqrt{5}, \quad \lambda(\{1,3\})=\sqrt{10}, \quad \lambda(\{2,3)\}=\sqrt{13},\\
&\lambda(\{1,2,3\})=\sqrt{14}.
\end{align*}
Therefore the Shapley value of the above game is
\begin{align*}
\phi_{1}(\lambda)&=\frac{1}{6}\left(2 \sqrt{14} + \sqrt{10} + \sqrt{5} - 3 - 2 \sqrt{13}\right),\\
\phi_{2}(\lambda)&=\frac{1}{6}\left(2 \sqrt{14} + \sqrt{13} + \sqrt{5} - 2 \sqrt{10}\right),\\
\phi_{3}(\lambda)&=\frac{1}{6}\left(2 \sqrt{14} + \sqrt{13} + \sqrt{10} + 3 - 2 \sqrt{5}\right).
\end{align*}

For $S=\{2,3\}$ the covariance matrix becomes
\begin{equation*}
\Sigma^{S}=
\begin{bmatrix}
1 & 0 \\
0 & 13
\end{bmatrix}
\end{equation*}
and the corresponding games is
\begin{equation*}
\lambda^{S}(\{1\})=1, \quad \lambda^{S}(S)=\sqrt{13}, \quad \lambda^{S}(\{1,S\})=\sqrt{14}.
\end{equation*}
The Shapley value of the above game is
\begin{align*}
\phi_{1}(\lambda^{S})&=\frac{1}{2}\left(1 +  \sqrt{14} - \sqrt{13}\right),\\
\phi_{S}(\lambda^{S})&=\frac{1}{2}\left(\sqrt{14} + \sqrt{13} - 1\right).
\end{align*}
We have 
\begin{equation*}
\phi_{S}(\lambda^{S}) \neq \phi_{2}(\lambda) + \phi_{3}(\lambda).
\end{equation*}
\end{example}

\subsection{A conjecture}

Given two vectors $\boldsymbol{x}, \boldsymbol{y} \in \mathbb{R}^{n}$ we say that $\boldsymbol{x}$ is majorized by $\boldsymbol{y}$ ($\boldsymbol{x} \prec \boldsymbol{y}$) if 
\begin{align*}
\sum_{i=k}^{n} x_{(i)} &\le \sum_{i=k}^{n} y_{(i)}\quad\text{for all }k\in\{1, \dots, n-1\},\\
\sum_{i=1}^{n} x_{i} &= \sum_{i=1}^{n} y_{i},
\end{align*}
where $x_{(1)} \le x_{(2)} \le \dots \le x_{(n)}$ is the increasing rearrangement of $\boldsymbol{x}$. The reader is referred to \citet{MarOlkArn:Springe2011} for properties of majorization. 

The following proposition shows that, for $n=2$, the normalized Shapley value of the variance game majorizes the corresponding normalized Shapley value of the standard deviation game.
\begin{proposition}
Consider a covariance matrix
\begin{equation*}
\Sigma =
\begin{bmatrix}
\sigma_{1}^{2} & \sigma_{12} \\
\sigma_{12} & \sigma_{2}^{2}
\end{bmatrix}.
\end{equation*} 
Call $\nu$ the corresponding variance game and $\lambda$ the corresponding standard deviation game. Then 
\begin{equation*}
\frac{1}{\phi_{1}(\lambda)+\phi_{2}(\lambda)} \boldsymbol{\phi}(\lambda) \prec 
\frac{1}{\phi_{1}(\nu)+\phi_{2}(\nu)} \boldsymbol{\phi}(\nu),
\end{equation*}
\end{proposition}

\begin{proof}
Assume, without any loss of generality, that $\sigma_{1} \le \sigma_{2}$.
We need to show that 
\begin{equation*}
\frac{\phi_{1}(\lambda)}{\phi_{1}(\lambda)+\phi_{2}(\lambda)} \ge 
\frac{\phi_{1}(\nu)}{\phi_{1}(\nu)+\phi_{2}(\nu)},
\end{equation*}
that is
\begin{equation*}
\frac{\sigma_{1}+\sqrt{\sigma_{1}^{2} + \sigma_{2}^{2} + 2 \sigma_{12}}-\sigma_{2} }{\sqrt{\sigma_{1}^{2} + \sigma_{2}^{2} + 2 \sigma_{12}}} \ge \frac{\sigma_{1}^{2}+\sigma_{12}}{\sigma_{1}^{2} + \sigma_{2}^{2} + 2 \sigma_{12}}.
\end{equation*}
After simple algebra, this corresponds to 
\begin{equation}\label{eq:frho}
(\sigma_{1}-\sigma_{2})\sqrt{\sigma_{1}^{2} + \sigma_{2}^{2} + 2 \rho\sigma_{1}\sigma_{2}} + \sigma_{2}^{2}+\rho\sigma_{1}\sigma_{2}\ge 0,
\end{equation}
where $\sigma_{12}=\rho\sigma_{1}\sigma_{2}$.

For $\rho=-1$, expression \eqref{eq:frho} becomes
\begin{equation*}
-\sigma_{1}^{2}+\sigma_{1}\sigma_{2}\ge 0,
\end{equation*}
and is therefore true.
Since the right hand side of \eqref{eq:frho} is increasing in $\rho$, we have the result.
\qed
\end{proof}

We conjecture the above result to be true for all $n \in \mathbb{N}$.

\begin{conjecture}
For any $n\times n$ covariance matrix $\Sigma$, if $\nu$ is the  corresponding variance game  and $\lambda$ the corresponding standard deviation game, then
\begin{equation}\label{eq:genmajor}
\frac{1}{\sum_{i=1}^{n}\phi_{i}(\lambda)} \boldsymbol{\phi}(\lambda) \prec 
\frac{1}{\sum_{i=1}^{n}\phi_{i}(\nu)} \boldsymbol{\phi}(\nu).
\end{equation}
\end{conjecture}

We have verified the conjecture numerically when the matrix $\Sigma$ is diagonal.
The program that verifies the conjecture was written in C and is based on the following consideration. 
Let 
\begin{align*}
S^{n-1} &= \left\{(\sigma_{1}, \dots, \sigma_{n})\in\mathbb{R}^{n}_{+}:\sum_{i=1}^{n}\sigma_{i}^{2}=1\right\}, \\
D_{n}&= \left\{(\sigma_{1}, \dots, \sigma_{n})\in\mathbb{R}^{n}_{+}:\sigma_{1} \le \sigma_{2} \le \dots \le \sigma_{n-1} \le \sigma_{n} \right\}, \\
M_{n} &= S^{n-1} \cap D_{n}.
\end{align*}
Because of the normalization factors in both sides of \eqref{eq:genmajor}, if the conjecture holds for each  diagonal matrix $\Sigma =\diag(\sigma_{1}^{2}, \sigma_{2}^{2}, \dots, \sigma_{n}^{2})$ then it holds also for the diagonal matrix $\Sigma'=\diag(\alpha\sigma_{\psi(1)}^{2}, \dots, \alpha\sigma_{\psi(n)}^{2})$, with $\alpha>0$ and $\psi$ any permutation of $(1, \dots, n)$.
Therefore, in order to verify the conjecture for any diagonal covariance matrix, it suffices to verify it for each $(\sigma_{1},\sigma_{2},\dots,\sigma_{n}) \in M_{n}$. 

The procedure works as follows. For a given number of players $n$ we extract $N$ independent normally distributed random vectors $\boldsymbol{Z}_j = (Z_{1,j},Z_{2,j},\dots,Z_{n,j}) \sim \mathcal{N}(\boldsymbol{0},I_n)$ where $I_n$ denotes the $n \times n$ identity matrix. 
It is well known \citep{Mul:CACM1959} that $\boldsymbol{X}_j = (X_{1,j},X_{2,j},\dots,X_{n,j})$, where $X_{i,j}= Z_{i,j}/||\boldsymbol{Z}_j||$, is uniformly distributed on $S^{n-1}$. 
Therefore $|\boldsymbol{X}_j|:=(|X_{1,j}|,|X_{2,j}|,\dots,|X_{n,j}|)$ is uniformly distributed on the intersection of $S^{n-1}$ and the nonnegative orthant of $\mathbb{R}^n$.
Call $\boldsymbol\sigma_j = (\sigma_{1,j},\sigma_{2,j},\dots,\sigma_{n,j})$  the nondecreasing rearrangement of $\boldsymbol X_j$. Then the points $\{\boldsymbol\sigma_j\}_{j=1}^N$ are independently uniformly distributed on the set $M_n$.
The procedure checks for each point  $\{\boldsymbol\sigma_j\}_{j=1}^N$ whether the $n-1$ inequalities given by the majorization conditions in \eqref{eq:genmajor} hold for the diagonal covariance matrix $\Sigma_j =\mathrm{diag}(\sigma_{1,j}^2, \sigma_{2,j}^2, \dots, \sigma_{n,j}^2)$. Conditions were tested when $n=3,4,5$ and $N = 10^{9}$.

\section{Computational aspects}\label{se:computational}

Theorem~\ref{th:shapley} shows that the Shapley value of the variance game can be easily computed in polynomial time. For each $i\in\{1, \dots, n\}$ the value $\phi_{i}(\nu)$ is just the sum of $n$ known covariances. We now want to frame this result in a more general framework concerning computational complexity of the Shapley value in suitable classes of games.

In a general coalition formation problem there are two main sources of computational complexity: the computation of each single coalition's value and how to distribute this value among the participants of each coalition.
The former appears when each coalition has to solve a hard optimization problem in order to compute its value.
The latter depends on the characteristic function of the game and on the solution concept.

In our setting, coalitions do not face any hard optimization problem to compute their values, and, among the solution concepts, we use the Shapley value as an allocation criterion.
Thus, the interesting question that we address is why the Shapley value of the variance game can be computed in an efficient way, whereas a similar method cannot be used for the standard deviation game.
The reason lies in the form of the characteristic function of the two games. The characteristic function of the variance game is easily decomposable in a sum of distinct functions, whereas the characteristic function of the standard deviation game is not decomposable due to the presence of the square root.
Therefore, to the best of our knowledge, all algorithms to compute \emph{exactly} the Shapley value of the standard deviation game are non-polynomial.

\citet{ConSan:P19NCAI2004} prove that the Shapley value is efficiently computable if the characteristic function of the game can be decomposed in a specific form.
In the following we  show that the characteristic function of the variance game respects their decomposition  requirements.
We first introduce the definition of \emph{decomposition} of a characteristic function.

\begin{definition}[\protect{\citet[Definition~4]{ConSan:P19NCAI2004}}]
The vector of characteristic functions $(\nu_{1},\nu_{2},\ldots,\nu_{T})$, with each $\nu_{t} : 2^{N} \to 	\mathbb{R}$, is a decomposition over $T$
issues of characteristic functions $\nu : 2^{N} \to \mathbb{R}$ if for any $J \subseteq N$, $\nu(J) = \sum_{t=1}^{T} \nu_{t}(J)$.
\end{definition}

The decomposition of the original characteristic function is particularly convenient if each $\nu_{t}$ restricts its focus on a subset of agents.

\begin{definition}[\protect{\citet[Definition~5]{ConSan:P19NCAI2004}}]\label{de:C}
We say that $\nu_{t}$ concerns only $C_{t} \subseteq N$ if $\nu_{t}(J_{1}) = \nu_{t}(J_{2})$ whenever $C_{t} \cap J_{1} = C_{t} \cap J_{2}$.
In this case, we only need to define $\nu_{t}$ over $2^{C_{t}}$.
\end{definition}

This representation shrinks the number of values from $2^{|N|}$ to $\sum_{t=1}^{T}2^{|C_{t}|}$, exponentially fewer than the original representation.
Notice that when $|C_{t}|$ is bounded by a small constant, the number of values is linear in $T$.

Now, the characteristic function of the variance game is represented by

\begin{equation*}
\nu(J)=\Var\left[\sum _{i\in J}X_{i}\right]=\sum _{i\in J}\sum _{j\in J}\Cov [X_{i},X_{j}].
\end{equation*}
Thus, for each set $J$, $\nu$ can be decomposed into $|J| (|J|+1)/2$ terms, considering that $\Cov [X_{i},X_{j}]=\Cov [X_{j},X_{i}]$.
Notice that for each set $J^{\prime} \subseteq J$ all the characteristic functions in $J^{\prime}$ are also present in $J$.
We can represent $T$ as the set of all pairs $(i,j)\in N\times N$ with $i\le j$. Consequently, $|T| = |N| (|N|+1)/2 < |N|^{2}$.

For each $(i,j) \in T$ we have
\[
 \nu_{(i,j)}(J) = 
  \begin{cases} 
  \Var[X_{i}]  & \text{if } J = (i,i),\\
   2\Cov(X_{i}, X_{j}) & \text{if } J = (i,j)  \text{ and } i\neq j,\\
   0 & \text{otherwise}.
  \end{cases}
\]
Therefore, by Definition~\ref{de:C} we know that each $\nu_t \in T$ concerns at most $2$ players, thus $|C_t| \leq 2$ for all $t \in T$. We can then apply the following theorem:

\begin{theorem}[\protect{\citet[Theorem~1]{ConSan:P19NCAI2004}}]
Suppose we are given a characteristic function
with a decomposition 
$\nu = \sum^{T}_{t=1}
\nu_{t}$,
represented as follows.
For each $t$ with $1\leq t \leq T$ we are given $C_{t} \subseteq N$, so that each $\nu_{t}$ concerns only $C_t$.
Each $\nu_{t}$ is flatly represented over $2^{C_{t}}$ , that is, for each $t$ with $1 \leq t \leq T$, 
we are given $\nu_{t} (J_{t})$ explicitly for each $J_{t} \subseteq C_{t}$.
Then (assuming that table lookups for the $\nu_{t}(J_{t})$, as well computations of factorials, multiplications and subtractions take constant time), we can compute the Shapley value of $\nu$ for any given agent in time
$O(\sum_{t=1}^{T} 2^{|C_{t}|})$, or less precisely $O(T \cdot 2^{\max_{t} |C_{t}| } )$. This holds
whether or not the characteristic function is increasing, and
whether or not it is superadditive.
\end{theorem}
This confirms the outcome of our Theorem~\ref{th:shapley}, that is, the Shapley value of the variance game is computable in polynomial time.

Similar computational aspects of the Shapley value based on decomposition ideas have been studied by 
\citet{IeoSho:P6CEC2005,IeoSho:P7CEC2006}.

\bibliographystyle{artbibst}
\bibliography{bibshapley}

\end{document}